\documentclass[10pt,twoside,smallextended]{siamltex}
\usepackage{amsfonts,epsfig}
\usepackage{amssymb}

\usepackage{xy}
\input{xy}
\xyoption{all} \xyoption{rotate}

\newtheorem{remark}[theorem]{Remark}

\newtheorem{example}[theorem]{Example}

\begin{document}

\bibliographystyle{plain}
\title{
Graph structure of commuting functions}

\author{
Peteris\ Daugulis\thanks{Department of Mathematics, Daugavpils
University, Daugavpils, LV-5400, Latvia (peteris.daugulis@du.lv).
} }

\pagestyle{myheadings} \markboth{P.\ Daugulis}{Graph structure of
commuting functions} \maketitle

\begin{abstract} The problem of finding graph structure of functions commuting with
a given function in terms of their functional graphs is
considered. Structure of functional graphs of commuting functions
is described. The problem is reduced to describing graph
homomorphisms of weakly connected components of functional graphs.
Four subcases with finite sets are considered: permutations
commuting with permutation, permutations commuting with a
function, functions commuting with a permutation and functions
commuting with a function.  For finite sets the number of
functions commuting with a given one and functions with extremal
properties are found. Results for finite sets are generalized to
the case of arbitrary sets where there are additional types of
functional graph components.
\end{abstract}


\section{Introduction} \label{intro-sec}

\subsection{The subject of study}

Composition of functions is an important binary operation in
function sets. This operation is so omnipresent and important in
mathematics, that its basic property - associativity has been
abstracted and accepted (due to associativity of set-theoretic
union and intersection as well) as a basic feature of algebraic
structures such as groups. Commutativity is the second most useful
property of algebraic structures, its importance originates from
commutativity of set-theoretic union and intersection. Functions
commuting with respect to the composition operation have been
studied for both purely theoretical and applied reasons. See
\cite{R} for an example of studies of commuting rational functions
dating back to the early 20th century.  Commutativity of linear
algebraic objects such as matrices with respect to multiplication
has been studied since Frobenius, see \cite{D}. Generalizations of
commuting functions, e.g. commuting matrices and operators, are
important in applications such as quantum physics.

In this paper we study graph structure of commuting functions and
the results involve graph models of functions - functional graphs.
The answer is well known for both functions being bijective in
finite sets. Permutations commuting with a given permutation form
a subgroup of the total permutation group, its algebraic structure
has been studied and generalized, see \cite{L}. The general case
does not seem to have been described in the literature, therefore
some further study and description of commuting functions seems
sufficiently motivated. These studies may provide additional links
between algebra and discrete mathematics. Our motivation and goal
of this paper is to fill this gap - to describe functions
commuting with a given function in terms of their functional
graphs with the functions being arbitrary. In graph-theoretic
terms this amounts to descriptions of homomorphisms of functional
graphs. Thus combinatorial problems of graph homomorphisms get an
algebraic interpretation.

\subsection{Structure and notations}

Basic notations and facts are reviewed in subsections \ref{13},
\ref{14}. The subsection \ref{15} contains the results for finite
sets, four main subcases are considered - permutations commuting
with a permutation (subsection \ref{16}), permutations commuting
with an arbitrary function (subsection \ref{17}), arbitrary
functions commuting with a permutation (subsection \ref{18}),
arbitrary functions commuting with an arbitrary function
(subsection \ref{19}). In each case functions having the minimal
number of commuting functions are described. In subsection
\ref{19} functions having one cycle and maximal number of
commuting functions are described. The section \ref{20} contains
generalizations results for infinite sets.

In this paper we denote the Cartesian product of sets $A_{1}\times
A_{2}\times ... \times A_{n}$ as
$\bigotimes\limits_{i=1}^{n}A_{i}$ (not to be confused with tensor
products).  Sequences (including pairs) of elements or sets are
denoted using square brackets. For example, the sequence having
elements $a_{1}$, $a_{2}$, ..., $a_{n}$, is denoted as $[a_{1},
a_{2}, ..., a_{n}]$. Cycles are denoted using brackets. The power
ser of the set $A$ is denoted as $2^{A}$. We use normal letters to
denote fixed objects and $\backslash mathcal$ letters to denote
objects as function values.

\subsection{Endofunctions, functional graphs and their mappings}\label{13} We denote the set of
endofunctions of a set $S$ by $\mathcal{F}un(S)$ and the set of
bijective $S$-endofunctions ($S$-permutations) by
$\mathcal{B}ij(S)$. Given a set $S$ and $f\in \mathcal{F}un(S)$ we
denote the set of all $S$-endofunctions commuting with $f$
($f$-centralizer) by $\mathcal{C}(f)$:
$$\mathcal{C}(f)=\{g\in \mathcal{F}un(S)|fg=gf\}.$$

We denote the set of all $S$-permutations commuting with $f$ by
$\mathcal{C}_{bij}(f)$:
$$\mathcal{C}_{bij}(f)=\{g\in \mathcal{B}ij(S)|fg=gf\}.$$

The graph $\Gamma$ with a vertex set $V$ and an edge set $E$ is
denoted by $\Gamma=(V,E)$, $\mathcal{V}(\Gamma)=V$,
$\mathcal{E}(\Gamma)=E$. Notation $\Delta\le \Gamma$ means that
$\Delta$ is a subgraph of $\Gamma$. We denote the directed cycle
$(V,E)$, where $V=\{x_{0},...,x_{n-1}\}$ (in this case and often
in this paper we denote indices of cycle elements as residues
$mod\ n$), $E=\bigsqcup\limits_{i=0}^{n-1} [x_{i},x_{i+1}]$ by
$(x_{0},...,x_{n-1})$.


Given a set $S$ and an endofunction $f\in \mathcal{F}un(S)$ we
define as usual the functional graph of $f$ (we call it \sl the
$f$-graph\rm): it is the directed graph $\Gamma(f)=(S,E_{f})$,
where $E_{f}=\bigsqcup\limits_{s\in S}[s,f(s)]$.

\begin{example}\label{ex1} Let $S=\{0,1,...,8\}$ - residues mod
$9$. Let $f:S\rightarrow S$, $f(x)\equiv x^2($mod\ $9)$. The
$f$-graph is shown below.

$$
\xymatrix{
3\ar[dr]& & 6\ar[dl] &8\ar[d]&2\ar[dr]&&&5\ar[dl]\\
& 0\ar@(dr,dl) && 1\ar@(dr,dl)&&4\ar@/_/[r]&7\ar@/_/[l]& \\
}
$$
\begin{center}
\

\

Fig.1.  - the $f$-graph for Example \ref{ex1}

\end{center}
\end{example}

Given two endofunctions $f$ and $g$ we can construct the weighted
\sl $(f,g)$-graph $\Gamma_{f,g}=(S,E_{f}\cup E_{g})$\rm\, where
edges of the sets $E_{f}$ and $E_{g}$ are weighted by $f$ and $g$,
respectively.

We remind the reader some basic graph theory definitions for
notational purposes. Suppose we are given two directed graphs
$\Gamma_{1}=(V_{1},E_{1})$ and $\Gamma_{2}=(V_{2},E_{2})$. We call
a function $f:V_{1}\rightarrow V_{2}$ \sl a graph
homomorphism,\rm\ denoted also as $f:\Gamma_{1}\rightarrow
\Gamma_{2}$, provided $[v,w]\in E_{1}$ implies $[f(v),f(w)]\in
E_{2}$. A graph homomorphism from $\Gamma$ to itself is called a
\sl $\Gamma$-endomorphism.\rm\

Given a directed graph $\Gamma$ we can forget orientations of its
arrows and get a undirected graph $\hat{\Gamma}$. Strictly
speaking, in general we get a multigraph since there may be pairs
of vertices with two directed edges between them. In case of
functional graphs this can happen only if there are cycles of
length $2$, the distinction between graphs and multigraphs does
not seem important for purposes of our paper. Two directed graphs
are called weakly isomorphic if the corresponding undirected
graphs are isomorphic. A \sl weakly connected component\rm\ of (a
directed graph) $\Gamma$ is an induced subgraph $\Delta\le \Gamma$
such that $\hat{\Delta}$ is a connected component of
$\hat{\Gamma}$. Clearly, $f:\Gamma_{1}\rightarrow \Gamma_{2}$ is a
graph homomorphism iff its restriction to every weakly connected
component is a graph homomorphism. Thus any graph homomorphism is
obtained by composing graph homomorphisms from weakly connected
components of the  domain. Homomorphisms from weakly connected
components can be constructed independently.

A directed graph $T$ is called a directed tree with the root $x$
provided 1) $\hat{T}$ is a tree and 2) there is a unique directed
path from any other vertex to $x$.

The set of homomorphisms (endomorphisms,  automorphisms etc.) from
$\Gamma$ to $\Delta$ is denoted by $\mathcal{H}om(\Gamma,\Delta)$
($\mathcal{E}nd(\Gamma)$, $\mathcal{A}ut(\Gamma)$).

If $\Gamma_{1}$ and $\Gamma_{2}$ are graphs with vertex weight
functions $w_{i}:\mathcal{V}(\Gamma_{i})\rightarrow k$, where $k$
is some weight set, then for $f:\mathcal{V}(\Gamma_{1})\rightarrow
\mathcal{V}(\Gamma_{2})$ to be a graph isomorphism it must satisfy
$w_{1}=w_{2}\circ f$.

Given a set $S$ and two $S$-endofunctions $f$ and $g$ we can
consider $g$ as a mapping for the graph $\Gamma(f)$ and vice
versa.

\begin{lemma} Let $S$ be a set. Let $f$ and $g$ be commuting
$S$-endofunctions.
Then

\begin{enumerate}

\item $fg=gf$ iff $g$ is a $\Gamma(f)$-endomorphism.

\item $fg=gf$ and $g\in \mathcal{B}ij(S)$ iff $g$ is an
$\Gamma(f)$-automorphism.

\end{enumerate}

\end{lemma}

\begin{proof}

1. If $[v,w]\in \mathcal{E}(\Gamma(f))$, then $w=f(v)$. We have
that $f(g(v))=(fg)(v)=(gf)(v)=g(w)$, therefore $[g(v),g(w)]\in
\mathcal{E}(\Gamma(f))$.

If $g\in \mathcal{E}nd(\Gamma(f))$, then $g(f(u))=f(g(u)))$,
therefore $fg=gf$.

2. By 1. $g\in \mathcal{E}nd(\Gamma(f))$.
 In the other direction, since $g\in \mathcal{C}(f)$ and $g\in
 \mathcal{B}ij(S)$,
 we have $g^{-1}\in \mathcal{C}(f)$, thus $g^{-1}\in \mathcal{E}nd(\Gamma(f))$
 and hence $g\in \mathcal{A}ut(\Gamma(f))$.
\end{proof}

\subsection{Weakly connected
components of functional graphs}\label{14}

Isomorphism types of weakly connected components of functional
graphs seem to be well known, we remind them below.

\subsubsection{Bijections}

Recall that if $f$ is a bijection, then weakly connected
components of the $f$-graph are directed cycles (called $f$-cycles
in this paper) or directed lines (called $f$-lines or infinite
$f$-cycles in this paper). We denote the isomorphism type of a
directed cycle on $n$ vertices by $Z_{n}$, we can assume that
$Z_{n}\simeq(\mathbb{Z}_{n},E_{n})$, where $\mathbb{Z}_{n}$ is the
set of residue classes $mod\ n$ and $[i,j]\in E_{n}$ iff
$i+1\equiv j(mod\ n)$. We denote the isomorphism type of directed
line by $L$, we assume that $L\simeq (\mathbb{Z},E_{\mathbb{Z}})$,
where $\mathbb{Z}$ is the set of integers and $[i,j]\in
E_{\mathbb{Z}}$ iff $i+1=j$. Thus for an arbitrary bijective
function $f$ we may assume that a weakly connected component of
the $f$-graph is isomorphic to $Z_{n}$ or $L$. This follows from
the observation, that in the $f$-graph every vertex has exactly
one incoming and outgoing edge.
\newpage

\subsubsection{Endofunctions on a finite set}

\paragraph{Pseudocycles}\ If $S$ is a finite set and $f\in
\mathcal{F}un(S)$ arbitrary, then the $f$-graph is a directed
pseudoforest -  weakly connected components of the $f$-graph are
traditionally called directed pseudotrees, see \cite{G}, but in
this paper we call them \sl $f$-pseudocycles.\rm\ We call a
directed graph \sl a pseudocycle\rm\ provided it contains exactly
one directed cycle (which may be a loop) and there is a unique
directed path from any other vertex to the closest cycle vertex,
thus any vertex of the directed cycle is a root of a directed tree
(which may consist of a single vertex). This description follows
from the observation, that in the $f$-graph every vertex has
exactly one outgoing edge. See \cite{F} for another description.
In terms of vertex weighted graphs a pseudocycle is a directed
cycle with vertices weighted by rooted directed trees. Thus we
think and denote a pseudocycle as \sl\ a tree cycle\rm\
$(T_{0},...,T_{m-1})$, where $T_{i}$ is isomorphic to a rooted
directed tree, two pseudocycles $P_{1}=(T_{0},...,T_{m-1})$ and
$P_{2}=(T'_{0},...,T'_{m-1})$ are isomorphic iff there is an
cyclic permutation $\zeta\in \Sigma_{m}$ of the sequence $P_{1}$
such that $\zeta(P_{1})\simeq P_{2}$, see below. Additionally,
conjugation by a permutation preserves the weak isomorphism type
(the cycle type in the special case of permutations).

\paragraph{Cyclic permutations of pseudocycles}\ Given a pseudocycle
$x=(x_{0},$ $...,$ $x_{m-1})$ with vertices from the multiset
$X=\{\{y_{1},...,y_{n}\}\}$ we call the vertex multiset
permutation $\sigma_{x}=(x_{0}...x_{m-1})$ (in cycle notation) \sl
the elementary shift of $x$.\rm\ We denote the set of all
permutations of $X$ by $\Sigma_{X}$. We call a permutation
$\rho\in \Sigma_{X}$ \sl cyclic permutation of $x$\rm\ if
$\rho=\sigma_{x}^{k}$ for some $k\in \mathbb{N}$. For any $x\in
X^{m}$ minimal $k\in \mathbb{N}$ such that $\sigma_{x}^{k}(x)=x$
is called \sl the order\rm\ of $x$, denoted by $ord(x)$. We call
$s_{x}=\frac{|x|}{ord(x)}$ \sl the index of $x$.\rm\
 We define two sequences (or pseudocycles) $x=(x_{0},...,x_{m-1})$ and
$y=(y_{0},...,y_{m-1})$ \sl cyclic isomorphic\rm\ if there is a
cyclic premutation $\zeta$ such that $\zeta(x)=y$.

\begin{example} If $x=(a,b,c,a,b,c)$, then
$\sigma_{x}(x)=(c,a,b,c,a,b)$, $ord(x)=3$, $s_{x}=2$.

\end{example}

\paragraph{Rooted directed trees of pseudocycles}\ We denote the directed
cycle of the pseudocycle $P$ by $\mathcal{Z}(P)$. Given $x\in
\mathcal{V}(P)$ we denote the rooted directed tree with the root
$x$ as $\mathcal{T}(x)$. We have that $\mathcal{T}(x)\cap
\mathcal{Z}(P)$ is $\{x\}$ if $x\in \mathcal{V}(\mathcal{Z}(P))$
or $\emptyset$ otherwise. By a (full) directed tree of a
pseudocycle $P$ we call $\mathcal{T}(z)$ with $z\in
\mathcal{V}(\mathcal{Z}(P))$: $\mathcal{T}(z)$ an induced subgraph
of $P$ such that $\mathcal{T}(z)\cap \mathcal{Z}(P)=\{z\}$ and
$\mathcal{V}(\mathcal{T}(z))$ contains all vertices of $P$ having
directed paths to $z$. We denote the pseudocycle having the
directed cycle $(z_{0},...,z_{m-1})$ and corresponding full
directed trees $T_{0},...,T_{m-1}$ by $(T_{0},...,T_{m-1})$.

In any directed tree we can introduce the corresponding tree
order: given two vertices $x,y$ of a directed tree we define $x\le
y$ provided there is a directed path from $y$ to $x$. Given a
directed tree $T$ denote by $\mathcal{D}_{i}(T)$ the set of
vertices being in distance $i$ to the root, thus
$\mathcal{V}(T)=\bigcup\limits_{i}\mathcal{D}_{i}(T)$. For
vertices of directed trees we define a \sl height function\rm\
$\phi$: $\phi(x)=i$ iff $x\in \mathcal{D}_{i}(T)$. Given a
directed tree $T$ with root $x$, we denote the root-truncated
graph $T\backslash \{x\}$ by $\widetilde{T}$.

\subsubsection{Endofunctions on an arbitrary set}\

If $S$ is infinite, then there are two additional types of weakly
connected components of functional graphs, which we call \sl
pseudolines (infinite pseudocycles)\rm\ and \sl pseudorays.\rm\

\paragraph{Pseudolines}

We call a directed graph a \sl\ pseudoline (infinite
pseudocycle)\rm\ provided it is isomorphic to a functional graph,
which contains at least one subgraph isomorphic to the directed
line. Note that the directed line $L$ is a special case of
pseudoline.

\begin{example}
Let $S=\mathbb{Z}$ and $f(x)=x+a$ for a fixed $a\in \mathbb{Z}$.
The $f$-graph has $a$ $f$-lines as weakly connected components.
\end{example}

We can note that in the pseudoline there is a unique directed path
from any vertex to the closest vertex of $\Lambda$, thus any
vertex of the directed line $\Lambda$ may be a root of a possibly
infinite directed tree.

\paragraph{Pseudorays}

We denote the isomorphism type of \sl directed ray\rm\ by $R$, we
may assume that $R\simeq(\mathbb{N},E_{\mathbb{N}})$, where
$\mathbb{N}$ is the set of natural numbers and $[i,j]\in
E_{\mathbb{N}}$ iff $i+1=j$. We call a functional graph a \sl
pseudoray\rm\ provided it contains a subgraph isomorphic to a
directed ray but no directed cycle or line.

\begin{example}\label{ex2}
Let $S=\mathbb{Z}$ and $f(x)=x^2$. The $f$-graph has countably
many pseudorays ($f$-rays) as weakly connected components. A
fragment of an $f$-ray is shown below.

$$
\xymatrix{
2\ar[r]& 4\ar[r]& 16\ar[r]&256\ar[r]& ... \\
-2\ar[ur]& -4\ar[ur]& -16\ar[ur]& -256\ar[ur]& ... \\
}
$$
\begin{center}
\

\

Fig.2.  - a pseudoray for Example \ref{ex2}

\end{center}

\end{example}

Similarly as in the case of pseudoline there is a unique finite
directed path from any vertex to the closest vertex of $\Lambda$,
thus any vertex of $\Lambda$ may be a root of a finite directed
tree. Note that in this case directed trees must be finite since
there are no directed line.

A pseudoline or a pseudoray may have more than one subgraph
isomorphic to a directed line or ray.

For pseudolines and pseudorays we can generalize notions of
directed trees with given roots etc. If $X$ is a directed line or
a directed ray in a functional graph and $x\in \mathcal{V}(X)$,
then $\mathcal{T}(x)$ is the maximal directed tree such that
$\mathcal{T}(x)\cap X=\{x\}$.

\section{Main results}

In this section we describe endofunctions $g\in \mathcal{F}un(S)$
commuting with a given endofunction $f\in \mathcal{F}un(S)$. In
terms of functional graphs we describe possible graph
homomorphisms of $f$-graphs. Descriptions are given as
correspondences $g\leftrightarrow [A, B, C, ...)$, where $A,B,...$
are mappings or substructures related to $S$, which are relatively
easy to describe. The first subsection \ref{15} deals with four
subcases for finite $S$. In the subsection \ref{20} the cases of
infinite $S$ are discussed.

\subsection{Finite sets}\label{15}

\subsubsection{Permutations commuting with a fixed
permutation}\label{16}

In this subsection $S$ is a finite set, $f$ is a permutation:
$f\in \mathcal{B}ij(S)$. We want to describe
$\mathcal{C}_{bij}(f)=\mathcal{C}(f)\bigcap \mathcal{B}ij(S)$ -
$f$-commuting permutations. This description seems to be well
known.

\begin{lemma}\label{1} Let $S$ be a finite set, let $f$ and $g$ be commuting $S$-endofunctions: $f g=g f$.
Let $Z=(x,f(x),...,$ $f^{k-1}(x))$ be an $f$-cycle of length $k$.

Then $g(Z)=(g(x),g(f(x)),...,g(f^{k-1}(x)))$ is an $f$-cycle of
length $k$.

\end{lemma}

\begin{proof} $fg=gf$ implies $g(f^{i}(x))=f^{i}(g(x))$ for all $x\in
S$ and $i\in \mathbb{Z}$. It follows that
$f(g(f^{i}(x)))=g(f^{i+1}(x))$ for $i\in \{0,...,k-2\}$. Since
$f^{k}(x)=x$ we have that $f(g(f^{k-1}(x)))=g(f^{k}(x))=g(x)$,
therefore $g(Z)=(g(x),$ $f(g(x)),$ $...,$ $f^{k-1}(g(x)))$ is an
$f$-cycle of length $k$. $g(Z)$ is independent of the choice of
$x\in Z$.
\end{proof}

\begin{remark} In terms of Lemma \ref{1} the isomorphic $f$-cycles $Z$ and $g(Z)$
may be equal or different (vertex disjoint). An $f$-commuting
permutation $g$ is determined on any $f$-cycle $Z$ by choosing
$g(x)$ in any $f$-cycle of length $|Z|$ for any fixed $x\in Z$.

\end{remark}

In the next theorem we describe $f$-commuting permutations.

\begin{theorem}\label{7} Let $S$ be a finite set, $f$ and $g$ - commuting $S$-permutations:  $f,g\in
\mathcal{B}ij(S)$, $fg=gf$. Denote the set of $f$-cycles of length
$i$ as $\mathbf{Z}_{i}=\bigsqcup\limits_{j=1}^{n_{i}}Z_{i,j}$,
where $Z_{i,j}$ denotes a cycle of length $i$. Define
$\mathbf{Z}=\bigsqcup\limits_{i}\mathbf{Z}_{i}$ For any $f$-cycle
$Z_{i,j}$ choose an element $x_{i,j}\in \mathcal{V}(Z_{i,j})$,
define the set $X=\bigsqcup\limits_{i,j}x_{i,j}$.

Then $g$ is bijectively defined by the pair $[\widetilde{g},
g|_{X}]$, where

\begin{enumerate}

\item $\widetilde{g}$ is a permutation $\mathbf{Z}\rightarrow
\mathbf{Z}$, such that
$\widetilde{g}(\mathbf{Z}_{i})=\mathbf{Z}_{i}$ and

\item $g|_{X}$ is the restriction of $g$ on $X$, $g|_{X}:
X\rightarrow S$, where $g(x_{i,j})\in
\mathcal{V}(\widetilde{g}(Z_{i,j}))$;
\end{enumerate}

\end{theorem}

\begin{proof}

By Lemma \ref{1} an $f$-commuting permutation $g$ is bijectively
determined on $f$-cycles of length $i$ by the sequence
$g(x_{i,1}),...,g(x_{i,n_{i}})$, where $g(x_{i,j})$ belongs to
some $f$-cycle of length $i$, denote this $f$-cycle by
$\widetilde{g}(Z_{i,j})$. For each $j$ $g(x_{i,j})$ determines
$g(Z_{i,j})$, $g$ is a permutation thus we get a permutation
$\widetilde{g}_{i}$ of $\mathbf{Z}_{i}$. Considering the set of
all $f$-cycles $\mathbf{Z}$ we can construct a permutation
$\widetilde{g}$, which fixes each $\mathbf{Z}_{i}$. For an
arbitrary set $X=\{x_{i,j}\}_{i,j}$ any function on $X$ given by a
set $g(x_{i,j})$, where $x_{i,j}$ belongs to a cycle of length $i$
and the corresponding function $\mathbf{Z}\rightarrow \mathbf{Z}$
is a permutation, the function $g$ can be extended to a
permutation of $S$ commuting with $f$.
\end{proof}

The next theorem describes cycle structure of $f$-commuting
permutations.

\begin{theorem} Let $S$ be a finite set, $f$ and $g$ - commuting $S$-permutations:  $f,g\in
\mathcal{B}ij(S)$, $fg=gf$. Denote the set of $f$-cycles of length
$i$ as $\mathbf{Z}_{i}=\bigsqcup\limits_{j=1}^{n_{i}}Z_{i,j}$. Let
$\widetilde{g}$ be defined as in Theorem \ref{7}.

Then
\begin{enumerate}

\item any cycle of $\widetilde{g}|_{\mathcal{V}(\mathbf{Z}_{i})}$
of length $k>1$ decomposes into $i$ $g$-cycles of length $k$;

\item any cycle of $\widetilde{g}|_{\mathcal{V}(\mathbf{Z}_{i})}$
of length $1$ corresponding a map $g_{Z}:Z\rightarrow Z$, where
$Z=\{x,f(x),...,f^{i-1}(x)\}$ is an $f$-cycle and
$g_{Z}(f^{k}(x))=f^{k+j}(x)$ ($0\le j\le i-1$) decomposes into
$GCD(i,j)$ $g$-cycles of length $\frac{i}{GCD(i,j)}$;

\item $|\mathcal{C}_{bij}(f)|=\prod\limits_{i=1}^{n}n_{i}!
i^{n_{i}}$.


\end{enumerate}

\end{theorem}

\begin{proof}

1. We have to find the cycle decomposition of the union of several
$f$-cycles of length $i$, which are cyclically permuted by $g$. If
$g$ cyclically permutes $k$ $f$-cycles $Z_{1},...,Z_{k}$ from
$\mathbf{Z}_{i}$,$k>1$, then due to the bijectivity of $g$ the
restriction of $g$  to the union
$\bigsqcup\limits_{j=1}^{k}\mathcal{V}(Z_{j})$ decomposes into $i$
cycles of length $k$.

2. We have to find the cycle decomposition of the $f$-cycle
$Z=(x,f(x),$ $...,$ $f^{i-1}(x))$ under $g_{Z}$, where
$g_{Z}(f^{k}(x))=f^{k+j}(x)$. $g_{Z}$ is the restriction of $g$ to
$Z$ and $Z$ is fixed by $g$,  the definition of $g_{Z}$ follows
from the commutativity condition. We have that every element of
form $f^{k}(x)$ lies in a $g$-cycle
$(f^{k}(x),f^{k+j}(x),...,f^{k+jc})$, where $c$ is the minimal
natural solution of the equation $jc\equiv 0(mod\ i)$ or the
equivalent equation $\frac{j}{GCD(i,j)}\cdot c\equiv 0(mod\
\frac{i}{GCD(i,j)})$. It follows that $c=\frac{i}{GCD(i,j)}$. Thus
the cycle length of every $g$-cycle of $Z$ is equal to $c$ and the
number of $g$-cycles is equal to $\frac{i}{c}=GCD(i,j)$.

3. For any $i\in \{1,...,n\}$ any $g$ is bijectively determined by
a permutation of $f$-cycles of length $i$ and sequence of elements
belonging to each such $f$-cycle thus the number of restrictions
of $f$-commuting permutations on $\mathbf{Z}_{i}$ is $n_{i}!\cdot
i^{n_{i}}$. For each $i$ the action of $g$ can be chosen
independently, therefore the statement follows by the product
rule.
\end{proof}

\begin{remark} Thus an $f$-commuting permutation permutes $f$-cycles of the same
length. It can be shown that, as a group, $\mathcal{C}_{bij}(f)$
can be expressed as a direct product of wreath products of certain
subgroups.

\end{remark}

We consider simplest questions in extremal combinatorics.
Obviously, if $|S|=n$, then $\max\limits_{f\in
\mathcal{B}ij(S)}|\mathcal{C}_{bij}(f)|=n!$ since
$|\mathcal{C}_{bij}(id)|=n!$. It is slightly less obvious to find
$\min\limits_{f\in \mathcal{B}ij(S)}|\mathcal{C}_{bij}(f)|$ and
permutations for which the minimum is achieved.

\begin{proposition}\label{24} Let $S$ be a finite set, $|S|=n\ge 3$. Then

\begin{enumerate}

\item $\min\limits_{f\in
\mathcal{B}ij(S)}|\mathcal{C}_{bij}(f)|=n-1$;

\item $|\mathcal{C}_{bij}(f)|=n-1$ iff $\Gamma(f)\simeq
Z_{1}\bigcup Z_{n-1}$.


\end{enumerate}

\end{proposition}

\begin{proof} This proposition is essentially the Problem 1 of a 2010 Russian student algebra olympiad,
we follow \cite{O}.

1. We first prove that $|\mathcal{C}_{bij}(f)|\ge n-1$. Let the
cycle decomposition of $f$ has $m$ fixed points and a set of
nontrivial cycles of lengths $n_{1},n_{2},...,n_{k}$. We have that
$m+\sum\limits_{i=1}^{k}n_{i}=n$. Let $\mathcal{C}_{0}\subseteq
\mathcal{C}_{bij}(f)$ be the set of $S$-permutations $g$ which
permute the fixed points and for which $\widetilde{g}$ is the
identity (every $f$-cycle is mapped to itself). By the product
rule we have  $|\mathcal{C}_{0}|=m!\prod\limits_{i=1}^{k}n_{k}\le
|\mathcal{C}_{bij}(f)|$.

Let $m\ge 2$. Then $|\mathcal{C}_{0}|\ge (1+(m-1))
\prod\limits_{i=1}^{k}(1+(n_{i}-1))$. The product has $2^{k+1}$
terms each at least $1$, there are $k+1$ linear terms of type
$n_{i}-1$. Thus $|\mathcal{C}_{0}|\ge
2^{k+1}+(m-1)-1+\sum\limits_{i=1}^{k}((n_{i}-1)-1)=n+(2^k+1-2k-2)$.
Since $2^{k+1}-2k-2\ge 0$, we have that $\mathcal{C}_{bij}(f)\ge
\mathcal{C}_{0}\ge n$.

Let $m=0$, then similarly
$\mathcal{C}_{0}=\prod\limits_{i=1}^{k}n_{i}=\prod\limits_{i=1}^{k}(1+(n_{i}-1))\ge
2^k+\sum\limits_{i=1}^{k}(n_{i}-2)=n+(2^k-2k)\ge n$.

Let $m=1$. Then $|\mathcal{C}_{0}|=\prod\limits_{i=1}^{k}n_{i}\ge
(n-1)+(2^k-2k)\ge n-1$.

For any $n\ge 3$ take $a_{n}$ be such that $\Gamma(a_{n})\simeq
Z_{1}\bigcup Z_{n-1}$. Then $|\mathcal{C}_{bij}(a_{n})|=n-1$.

2. If $m\ge 2$ or $m=0$, then $|\mathcal{C}_{bij}(f)|\ge n$.

Suppose $m=1$ and there are at least two nontrivial $f$-cycles. If
there are two nontrivial $f$-cycles of equal length then
$f$-commuting permutations can permute these $f$-cycles and thus
$|\mathcal{C}_{bij}(f)|\ge n$. If all nontrivial $f$-cycles have
different lengths then
$|\mathcal{C}_{bij}(f)|=\prod\limits_{i=1}^{k}\ge
\sum\limits_{i=1}^{k}n_{k}=n-1$. If $1<n_{1}<n_{2}<...<n_{k}$ then
a strict inequality
$\prod\limits_{i=1}^{k}n_{k}>\sum\limits_{i=1}^{k}n_{i}$ is true,
this has been proven elsewhere. Thus $|\mathcal{C}_{bij}(f)|=n-1$
only for the cycle type $Z_{1}\bigcup Z_{n-1}$.
%
\end{proof}

\subsubsection{Permutations commuting with a fixed
endofunction}\label{17}

In this subsection we describe permutations commuting with an
arbitrary fixed endofunction $f$ given on a finite set $S$ -
$\mathcal{C}_{bij}(f)=\mathcal{C}(f)\cap \mathcal{B}ij(S)$.

\begin{lemma}\label{3} Let $S$ be a finite set, $f\in \mathcal{F}un(S)$,
$g$ - an $f$-commuting permutation: $g\in \mathcal{B}ij(S)$,
$fg=gf$. Let $P=(T_{0},...,T_{m-1})$ be an $f$-pseudocycle.

Then $g(P)$ is an $f$-pseudocycle, which is cyclic isomorphic to
$P$.

\end{lemma}

\begin{proof} $g$ is an $\Gamma(f)$-automorphism, therefore $P\simeq g(P)$.
$\mathcal{Z}(P)$ is mapped by $g$ to $\mathcal{Z}(g(P))$ by a
cyclic permutation. Each directed tree $\mathcal{T}(z)$ of $P$ is
mapped isomorphically to the directed tree $\mathcal{T}(g(z))$.
\end{proof}

\begin{lemma}\label{5} Let $S$ be a finite set, $f\in \mathcal{F}un(S)$,
$g$ - an $f$-commuting permutation: $g\in \mathcal{B}ij(S)$,
$fg=gf$. $P=(T_{0},...,T_{m-1})$ is an $f$-pseudocycle.

Then
 $\zeta: \mathcal{V}(P)\rightarrow \mathcal{V}(P)$ is a $P$-automorphism
if and only if $\zeta$ is a cyclic permutation on $\mathcal{Z}(P)$
and $\zeta(T_{i})\simeq T_{i}$ for any $i\in \{0,..,m-1\}$.

\end{lemma}

\begin{proof}
$\zeta$ must cyclically permute vertices of $\mathcal{Z}(P)$ since
it is the only oriented cycle of $P$. $\zeta$ must send each
directed tree $T_{i}$ to an isomorphic directed tree.
\end{proof}

Now we describe $f$-commuting permutations with an arbitrary $f$.

\begin{theorem}\label{33} Let $S$ be a finite set, $f\in \mathcal{F}un(S)$, $g$ - an $f$-commuting permutation:
 $g\in \mathcal{B}ij(S)$, $fg=gf$. Denote the set
of $f$-pseudocycles with a tree cycle $T$ as
$\mathbf{P}_{T}=\bigsqcup\limits_{j=1}^{n_{T}}P_{T,j}$, $P_{T,j}$
denotes an $f$-pseudocycle with a tree cycle $T$. Denote
$\mathbf{P}=\bigsqcup\limits_{T}\mathbf{P}_{T}$. For any
$f$-pseudocycle $P_{T,j}$ choose an element $x_{T,j}\in
\mathcal{V}(Z(P_{T,j}))$, let $X_{T}=\bigsqcup\limits_{j}x_{T,j}$,
$X=\bigsqcup\limits_{T}X_{T}$.

Then for every $T$ restriction $g|_{\mathcal{V}(\mathbf{P}_{T})}$
is bijectively determined by the triple $\tau=[\widetilde{g}_{T},
g|_{X_{T}}, A_{T}]$, where

\begin{enumerate}

\item $\widetilde{g}_{T}$ is a permutation of the set of
$f$-pseudocycles $\mathbf{P}_{T}$;

\item $g|_{X_{T}}$ is the restriction of $g$ on $X_{T}$,
$g|_{X_{T}}: X_{T}\rightarrow S$, where
$\mathcal{T}(f^{k}(x_{T,j}))\simeq \mathcal{T}(g(f^{k}(x_{T,j})))$
for all $j\in \{1,...,n_{T}\}$, $k\in \{0,...,|T|-1\}$ ($|T|$
denotes the length of $\mathcal{Z}(T)$);

\item $A_{T}=[\alpha_{jk}]_{j=1,k=0}^{n_{T},|T|-1}$, where
$\alpha_{jk}: \mathcal{T}(f^{k}(x_{T,j}))\rightarrow
\mathcal{T}(g(f^{k}(x_{T,j})))$ is an isomorphism of directed
trees ($A_{T}$ is a two dimensional array of graph isomorphisms).

\end{enumerate}

\end{theorem}

\begin{proof} By Lemma \ref{3} and Lemma \ref{5} any $f$-pseudocycle $P_{T,j}$ is necessarily mapped by the
isomorphism $g$ to an $f$-pseu\-do\-cyc\-le with cyclic isomorphic
sequence of directed trees, i.e.
$$\mathcal{T}(f^{k}(x_{T,j}))\simeq \mathcal{T}(g(f^{k}(x_{T,j}))),\ k\in
[0,...,|T|-1].$$ Every directed tree is mapped isomorphically to
the corresponding directed tree, which gives the array of directed
tree isomorphisms $\alpha_{j,k}$. For each $T$ and $j$
$g(x_{T,j})$ determines $\mathcal{Z}(g(P_{T,j}))$, by adding the
action of $g$ on directed trees by isomorphisms $\alpha_{j,k}$ we
get the restriction of $g$ on $\mathbf{P}_{T}$.
\end{proof}

The next theorem describes cycle structure of $f$-commuting
permutations for an arbitrary $f$.

\begin{theorem} Let $S$ be a finite set, $f\in \mathcal{F}un(S)$, $g$ - an $f$-commuting permutation:
 $g\in \mathcal{B}ij(S)$, $fg=gf$. Denote the set
of $f$-pseudocycles with a tree cycle $T$ as
$\mathbf{P}_{T}=\bigsqcup\limits_{j=1}^{n_{T}}P_{T,j}$, $P_{T,j}$
denotes an $f$-pseudocycle with a tree cycle $T$. Let
$\widetilde{g}_{T}$ be defined as in Theorem \ref{33}.

Then

\begin{enumerate}

\item each cycle of $\widetilde{g}_{T}$ of length $k>1$ decomposes
into $g$-cycles of length $k$;

\item each cycle of $\widetilde{g}_{T}$ of length $1$
corresponding to a map $g_{P}:\mathcal{V}(P)\rightarrow
\mathcal{V}(P)$, where $P$ is an $f$-pseudocycle,
$\mathcal{Z}(P)=\{x,f(x),...,f^{i-1}(x)\}$ is an $f$-cycle and
$g_{P}(f^{k}(x))=f^{k+j}(x)$ ($0\le j\le i-1$) decomposes into
$g$-cycles of length $\frac{i}{GCD(i,j)}$;

\item $|\mathcal{C}_{bij}(f)|=\Big(\prod\limits_{T}n_{T}!
s_{T}^{n_{T}}\Big)\cdot \Big(\prod\limits_{z\in
\mathcal{Z}(\mathcal{V}(f))}|\mathcal{A}ut(\mathcal{T}(z))|\Big)$,
where $s_{T}$ is the index of $T$ ($\mathcal{Z}(\mathcal{V}(f))$
is the union of all $f$-cycles of $\mathcal{V}(f)$).
\end{enumerate}

\end{theorem}

\begin{proof}

1. We have to find the cycle decomposition of the union of several
$f$-pseudocycles each having vertex sets of size
$|\mathcal{V}(P_{T,1})|$, which are cyclically permuted by $g$. If
$g$ cyclically permutes $k$ $f$-pseudocycles $P_{1},...,P_{k}$
from $\mathbf{P}_{T}$, then due to the bijectivity of $g$ the
restriction of $g$  to the union $\bigcup_{j=1}^{k}P_{j}$
decomposes into $|\mathcal{V}(P_{T,1})|$ cycles of length $k$.

2. Proved similarly to 3. of Theorem \ref{7} . The cycle
decomposition of $\mathcal{Z}(P)$ has $GCD(i,j)$ $g$-cycles of
length $\frac{i}{GCD(i,j)}$. It induces cycle decomposition of
$P_{T}$ into $g$-cycles of the same length. The exact number of
these $g$-cycle is not given here.

\medskip

3. $\mathbf{P}_{T}$ are permuted independently, for each tree
cycle $T$ the number of permutations of $\mathbf{P}_{T}$ is
$n_{T}!$, the number of automorphisms of $\mathcal{Z}(T)$ is
$s_{T}$. Hence the number of different restrictions of commuting
permutations on cycles of $\mathbf{P}_{T}$ is $n_{T}!\cdot
s_{T}^{n_{T}}$, each directed tree of the tree cycle $T$ can be
twisted by an automorphism, formula follows by the product rule.
\end{proof}

\begin{remark} Thus an $f$-commuting permutation $g$
independently permutes $f$-pseudocycles having isomorphic cycles
of directed trees. From Lemma \ref{5} it follows that if the tree
cycle $T$ is such that it is fixed up to isomorphism by cyclic
permutations of order, which is a divisor of the total cycle
order, then corresponding $f$-pseudocycles may allow more than one
restriction of $f$-commuting permutation on their cycles (for each
pair of pseudocycles).

Additionally directed trees can be independently twisted by
automorphisms.
\end{remark}

\begin{example} \label{ex5} Let $S=\{0,...,7\}$ and the $f$-graph is given in
Figure 3 below

$$
\xymatrix{
4\ar[dr]&5\ar[d]&1\ar[dr]&6\ar[d]&7\ar[dl]\\
&0\ar[ur]&&2\ar[dl]& \\
&&3\ar[ul]&&\\ }
$$
\begin{center}

Fig.3.  - the $f$-graph for Example \ref{ex5}

\end{center}

In this case there is a single directed cycle $T=(0,1,2,3)$,
$n_{T}=1$, $s_{T}=2$. There are $2$ directed trees each having $2$
automorphism. Thus $|\mathcal{C}_{bij}(f)|=2\cdot 2^2=8$.

\end{example}

\begin{proposition}\label{23} Let $S$ be a finite set. Then

\begin{enumerate}

\item $\min\limits_{f\in
\mathcal{F}un(S)}|\mathcal{C}_{bij}(f)|=1$;

\item $|\mathcal{C}_{bij}(f)|=1$ iff any two weakly connected
components of $\Gamma(f)$ are not isomorphic and for every weakly
connected component $T=(T_{0},...,T_{m-1})$ of $\Gamma(f)$ two
conditions hold: a) index of $T$ is equal to $1$ and b)
$|Aut(T_{i})|=1$ for every directed tree.

\end{enumerate}

\end{proposition}

\begin{proof}

1. Let $|S|=n$. Take $f$ be such that $\Gamma(f)$ has one weakly
connected component and the directed cycle of length $n-1$. Then
$\mathcal{C}_{bij}(f)=\{id\}$.

2. $|\mathcal{C}_{bij}(f)|=1$ ($f$ commutes only with $id$) iff
the conjunction of three conditions holds: a) no two weakly
connected components of $\Gamma(f)$ are isomorphic, since
otherwise they could be permuted, b) the index of any weakly
connected component is $1$, since otherwise some component could
be mapped to itself and c) there are no nontrivial automorphisms
of directed trees of any component, since otherwise the function
that would fix all other vertices and twist a directed tree by a
nontrivial automorphism would produce a notrivial commuting
function.
\end{proof}

\subsubsection{Endofunctions commuting with a fixed
permutation}\label{18}

In this subsection we describe arbitrary endofunctions commuting
with a permutation $f$ of a finite set $S$.

\begin{lemma}\label{2} Let $S$ be a finite set, $f$ - a permutation on $S$, $g$ - an arbitrary
$f$-commuting endofunction: $g\in \mathcal{F}un(S)$, $fg=gf$. Let
$Z=(x,f(x),$ $...,$ $f^{k-1}(x))$ be an $f$-cycle of length $k\in
\mathbb{N}$.

Then there is $l\in \mathbb{N}$ such that $l|k$ and
$g(Z)=(g(x),g(f(x)),...,g(f^{l-1}(x)))$ is an $f$-cycle of length
$l$. The $f$-cycle $g(Z)$ is determined by the $g$-image of any
element of $Z$.

\end{lemma}

\begin{proof} $fg=gf$ implies $g(f^{i}(x))=f^{i}(g(x))$ for all $x\in
S$ and $i\in \mathbb{N}$. Suppose that $g(x)$ belongs to an
$f$-cycle of length $l$.  Since $f^{k}(x)=x$ we must have
$g(x)=f^{k}(g(x))$. It follows, that $l|k$ and for any $n\in
\mathbb{N}$ we have $(gf^{n})(x)=(f^{n(mod\ l)}g)(x)$, where $0\le
n(mod\ l)<l$.

If $x\in Z$ and $y=f^{l}(x)$, then $g(y)=f^{l}(g(x))$.
\end{proof}

\begin{remark} Lemma \ref{2} amounts to the fact that a homomorphic image of
a directed cycle of length $k$ is a directed cycle of length $l$
with $l|k$. For example, if $f(x)=x$ and $fg=gf$, then
$f(g(x))=g(x)$.

\end{remark}

\begin{example}\label{ex4} A cycle $(0,1,2,3)$ can be
homomorphically mapped by $g$ to the cycle $(4,5)$ as shown below.
\newpage

$$
\xymatrix{
0\ar[d]\ar@/^/[rrr]^g&2\ar[d]\ar[rr]_g&&4\ar@/_/[d]\\
1\ar[ur]\ar@/_/[rrr]_g&3\ar[ul]\ar[rr]^g&&5\ar@/_/[u]\\
}
$$
\begin{center}

\

Fig.4.  - the $f$-graph for Example \ref{ex4}

\end{center}
\end{example}

\begin{theorem}\label{6} Let $S$ be a finite set, $f$ - a permutation,  $g$ - an arbitrary $f$-commuting
endofunction: $fg=gf$. Denote the set of $f$-cycles of length $i$
as $\mathbf{Z}_{i}=\bigsqcup\limits_{j=1}^{n_{i}}Z_{i,j}$, denote
$\mathbf{Z}=\bigsqcup\limits_{i}\mathbf{Z}_{i}$. For any $f$-cycle
$Z_{i,j}$ choose an element $x_{i,j}\in \mathcal{V}(Z_{i,j})$,
denote $X=\bigsqcup\limits_{i,j}x_{i,j}$.

Then $g$ is bijectively determined by the pair
$[\widetilde{g},g|_{X}]$, where

\begin{enumerate}

\item $\widetilde{g}$ is a function
$\widetilde{g}:\mathbf{Z}\rightarrow\mathbf{Z}$ such that
$|\mathcal{V}(\widetilde{g}(Z_{i,j}))|$ divides $i$ for any
$i$,$j$.

\item $g|_{X}$ is the restriction of $g$ on $X$, $g|_{X}:
X\rightarrow S$, where $g(x_{i,j})\in
\mathcal{V}(\widetilde{g}(Z_{i,j}))$.

\end{enumerate}

\end{theorem}

\begin{proof} By Lemma \ref{2} an $f$-cycle $Z_{i,j}$ of length $i$ can be homomorphically
mapped only to an $f$-cycle of length $l$, where $l|i$, this
defines $\widetilde{g}$. For any $f$-cycle $\widetilde{g}$ is
determined by the $g$ image of one element, say, $x_{i,j}\in
Z_{i,j}$. Images of $f$-cycles can be chosen independently.
\end{proof}

\begin{theorem}\label{42} Let $S$ be a finite set, $f$ - a permutation,  $g$ - an arbitrary $f$-commuting
endofunction: $fg=gf$. Denote the set of $f$-cycles of length $i$
as $\mathbf{Z}_{i}=\bigsqcup\limits_{j=1}^{n_{i}}Z_{i,j}$, denote
$\mathbf{Z}=\bigsqcup\limits_{i}\mathbf{Z}_{i}$. Let
$\widetilde{g}$ be defined as in Theorem \ref{6}.

Then
\begin{enumerate}

\item Any $\widetilde{g}$-cycle of length $k>1$ permuting
$f$-cycles of length $i$ decomposes into $i$ $g$-cycles of length
$k$, any $\widetilde{g}$-cycle of length $1$ corresponding to a
$f$-cycle of length $i$ fixed by $g$ decomposes into $GCD(i,j)$
$g$-cycles of length $\frac{i}{GCD(i,j)}$ for some $j$.

\item If $g(Z_{i,j})=Z_{k,l}$ with $k|i$, then the restriction of
$g$ on $Z_{i,j}\cup Z_{k,l}$ decomposes into a forest of $k$
directed trees of size $\frac{i}{k}+1$;

\item $|\mathcal{C}(f)|=\prod\limits_{i}(\sum\limits_{d|i}n_{d}
d)^{n_{i}}$, where $n_{a}$ is the number of $f$-cycles of length
$a$.

\end{enumerate}

\end{theorem}

\begin{proof}
1. Proved similarly to 2. and 3. of Theorem \ref{7} .

2. If $g(Z_{i,j})=Z_{k,l}$ with $k|i$, then the inverse image of
any element of $\mathcal{V}(Z_{k,l})$ contains $\frac{i}{k}$
elements.

3. Any element of $\mathcal{V}(\mathbf{Z}_{i})$ can be mapped to
an admissible $f$-cycle of length $d$ in $d$ ways. Therefore by
the sum rule the total number of admissible mappings of one vertex
is $\sum\limits_{d|i}n_{d} d$. The final formula follows by the
product rule since images of $f$-cycles can be constructed
independently.
\end{proof}

\begin{remark} The function $\widetilde{g}$ described in Theorem \ref{2} defines a pseu\-do\-fo\-rest on
$\mathbf{Z}$. Thus to define a function commuting with a
permutation $f$ we need to choose an appropriate endofunction of
$\mathbf{Z}$ and a set of representatives for $f$-cycles.
\end{remark}

\begin{remark} Combining proposals 2. and 3. of Theorem \ref{6} we
can deduce the pseudotree decomposition of $g$: any
$\widetilde{g}$-pseudocycle $P$ decomposes into
$g$-pseu\-do\-cyc\-les, which can be recovered starting from the
decomposition of $\mathcal{Z}(P)$.

\end{remark}

\begin{example}
Let $f$ be a permutation having cycle type $4^{x}3^{y}2^{z}1^{t}$
(where $i^{j}$ means, that there are $j$ cycles of length $i$). We
can check, that $$|\mathcal{C}(f)|=t^{t}(t+2z)^{z}(t+3y)^{y}
(t+2z+4x)^{x}.$$
\end{example}

\begin{proposition}\label{41} Let $S$ be a finite set, $|B|=n$, $f\in \mathcal{B}ij(f)$. Then

\begin{enumerate}

\item $\min\limits_{f\in \mathcal{B}ij(S)}|\mathcal{C}(f)|=n$;

\item $|\mathcal{C}(f)|=n$ iff $\Gamma(f)\simeq Z_{n}$ or
$\Gamma(f)\simeq Z_{1}\bigcup Z_{n-1}$.

%

\end{enumerate}

\end{proposition}

\begin{proof}

1. From Proposition \ref{24} we know that $|\mathcal{C}(f)|\ge
|\mathcal{C}_{bij}|\ge n-1$ and the minimum of
$|\mathcal{C}_{bij}(f)|$ is achieved on functional graphs
isomorphic to $Z_{1}\bigcup Z_{n-1}$. Let $f$ be such that
$\Gamma(f)\simeq Z_{1}\bigcup Z_{n-1}$, then
$\mathcal{C}(f)=\{f^k|f\in \mathbb{Z}\}\bigcup \epsilon$, where
$\epsilon$ is the function which sends every vertex to the
$f$-fixed point. We have that $|\mathcal{C}(f)|=n$.

2. If $\Gamma(f)\simeq Z_{n}$ then
$|\mathcal{C}(f)|=|\mathcal{C}_{bij}(f)|=n$. If $\Gamma(f)\simeq
Z_{1}\bigcup Z_{n-1}$ then it was shown just above that
$|\mathcal{C}(f)|=n$.

Now consider the other implication. We use arguments given in
\cite{O}. Let $|\mathcal{C}(f)|=n$. As in the proof of Proposition
\ref{24} let the cycle decomposition of $f$ has $m$ fixed points
and a set of nontrivial cycles of lengths $n_{1},n_{2},...,n_{k}$,
$n_{i}\ge 2$ for all $i$.

Let $m=0$. Then $|\mathcal{C}(f)|\ge |\mathcal{C}_{bij}(f)|\ge
\prod\limits_{i=1}^{k}n_{i}\ge
\prod\limits_{i=1}^{k}(1+(n_{i}-1))\ge
2^k+\sum\limits_{i=1}^{k}((n_{i}-1)-1)=n+(2^k-2k)$. If $k>2$, then
$|\mathcal{C}(f)|\ge |\mathcal{C}_{bij}(f)|>n$. If $k=2$, then
$n_{1}n_{2}=n=n_{1}+n_{2}$ only if $n_{1}=n_{2}=2$ and $n=4$. In
this case $|\mathcal{C}_{bij}(f)|=8>4=n$. We are left with the
only possible choice $k=1$. If $\Gamma(f)\simeq Z_{n}$, then
$|\mathcal{C}(f)|=n$.

Let $m=1$. Then $|\mathcal{C}(f)|\ge |\mathcal{C}_{bij}(f)|\ge
n-1$ and the bound is reached for $\Gamma(f)\simeq Z_{1}\bigcup
Z_{n-1}$.
%
\end{proof}

\subsubsection{Arbitrary endofunctions commuting with a fixed arbitrary endofunction}\label{19}

Finally in this subsection we consider the general case for a
finite set $S$ -  $\mathcal{C}(f)$ with $f\in \mathcal{F}un(S)$.
In this case we describe restrictions of commuting functions on
individual weakly connected components.

\begin{theorem}\label{9} Let $S$ be a finite set, $f$ and $g$ - commuting $S$-sndofunctions:
$fg=gf$. Denote the set of directed cycles of length $i$ of the
$f$-pseudoforest as
$\mathbf{Z}_{i}=\bigsqcup\limits_{j=1}^{n_{i}}Z_{i,j}$. Denote
$\mathbf{Z}=\bigsqcup\limits_{i}\mathbf{Z}_{i}$.

Then the restriction of $g$ on $\mathcal{V}(\mathbf{Z})$
determines a function
$\widetilde{g}:\mathbf{Z}\rightarrow\mathbf{Z}$, such that
$|\widetilde{g}(Z_{i,j})|$ divides $i$ for all $i$,$j$.

\end{theorem}

\begin{proof} Similarly to Theorem \ref{6}. \end{proof}

The next two lemmas and Theorem \ref{11} deal with images of
directed trees under graph homomorphisms.

\begin{lemma}\label{10} Let $S$ be a finite set, $f$ and $g$ - commuting endofunctions:
$f g=g f$. Let $P$ be an $f$-pseudocycle. Let $Z$ be the union of
$f$-cycles. Let $x\in \mathcal{V}(P)$, $y\in f^{-1}(x)$.

Then
\begin{enumerate}

\item if $g(x)\in \mathcal{V}(Z)$, then either $g(y)$ and $g(x)$
belong to the same $f$-cycle, or $\phi(g(y))=1$;

\item if $\phi(g(x))=i>0$, then $\phi(g(y))=i+1$.

\end{enumerate}

\end{lemma}

\begin{proof} Follows from commutativity of $f$ and $g$. \end{proof}

\begin{lemma}\label{4} Let $S$ be a finite set, $f$ and $g$ - commuting endofunctions:
$f g=g f$. Let $P$ be an $f$-pseudocycle, $T$ is the (full)
directed tree of $P$ with the root $z$.

Then there is $\mathcal{A}\subseteq \mathcal{V}(T\backslash
\{z\})$, such that:

\begin{enumerate}

\item[1)] for any $x\in \mathcal{V}(T)$ we have that
$\phi(g(x))=1$ iff $x\in \mathcal{A}$;

\item[2)] there are no distinct elements $x\in \mathcal{A}$, $y\in
\mathcal{A}$, such that $x<y$ in the tree order;

\item[3)] for any $x\in \mathcal{A}$, if $x\le y$, then $g(x)$ and
$g(y)$ are in the same directed tree of $g(P)$ and $g(x)\le g(y)$.

\end{enumerate}

\end{lemma}

\begin{proof} We construct $\mathcal{A}$ using 1) as its definition. Alternatively we construct it iteratively
considering images under $g$ of the sequence of subsets
$[\mathcal{D}_{1}(T), \mathcal{D}_{2}(T), ...]$.

Consider $\mathcal{D}_{1}(T)$. For any $x\in \mathcal{D}_{1}(T)$
either $g(x)\in \mathcal{V}(\mathcal{Z}(g(P)))$ or $\phi(g(x))=1$.
Define $\mathcal{A}(1)=\{x\in \mathcal{D}_{1}(T)|\phi(g(x))=1\}$.
If $y\in \mathcal{V}(T)$ is such that $y\ge x$ for some $x\in
\mathcal{A}_{1}$, then $\phi(g(y))>1$ and $g(y)\ge g(x)$.

Consider $\mathcal{D}_{2}(T)$. For any $x\in \mathcal{D}_{2}(T)$
either $g(x)\in \mathcal{V}(\mathcal{Z}(g(P)))$, $\phi(g(x))=1$ or
$\phi(g(x))=2$ (this happens if $x>x_{1}$ for some $x_{1}\in
\mathcal{A}(1)$). Define $\mathcal{A}(2)=\{x\in
\mathcal{D}_{2}(T)|\phi(g(x))=1\}$. If $y\in \mathcal{V}(T)$ is
such that $y\ge x$ for some $x\in \mathcal{A}(2)$, then
$\phi(g(y))>1$ and $g(y)\ge g(x)$.

We continue this process until we reach the maximal $k$ such that
$\mathcal{D}_{k}(T)\ne \emptyset$. Define
$\mathcal{A}=\bigsqcup\limits_{h\ge 1}\mathcal{A}(h)$. Statement
$1)$ follows by construction, statements $2)$,$3)$ follow by Lemma
\ref{10}.
\end{proof}

Let $T$ be a directed tree with the root $z$. We call a $T$-vertex
subset $A\subseteq \mathcal{V}(T\backslash \{z\})$ \sl
incomparable vertex subset\rm\ provided there are no two distinct
$x,y\in A$ such that $x<y$ in the tree order. Denote by $Inc(T)$
the set of all incomparable vertex subsets of $T$.

\begin{example}\label{ex9} Let $T$ be the directed tree given
below:

$$
\xymatrix{
3\ar[dr]&&4\ar[dl]&5\ar[d]\\
&1\ar[dr]&&2\ar[dl]\\
&&0&\\
}
$$
\begin{center}

\

Fig.5.  - the $f$-graph for Example \ref{ex9}

\end{center}

Then $Inc(T)$ contains $\emptyset$,  five $1$-element subsets
$1,2,3,4,5$,  seven $2$-element subsets $\{1,2\},\{1,5\},\{2,3\},
\{2,4\},\{3,4\},\{3,5\},\{4,5\}$ and one $3$-element subset
$\{3,4,5\}\}$.

\end{example}

\begin{remark} For any directed tree the minimal number of
elements of an incomparable vertex subset is $0$, but the maximal
number of elements is the number of vertices of indegree $0$ (\sl
leaves\rm).

\end{remark}

\begin{theorem} \label{11} Let $S$ be a finite set, $f$ and $g$ - commuting endofunctions:
$fg=gf$.

Let $P$ be a $f$-pseudocycle of cycle length $i$ with the
$f$-cycle $Z=(z_{0},...,z_{i-1})$.

Then the restriction $g|_{P}$ is bijectively determined by the
triple $\tau=[g(z_{0}), [\mathcal{A}], [\Phi]]$, where

\begin{enumerate}

\item $[\mathcal{A}]$ is a sequence of incomparable vertex
subsets: $[\mathcal{A}]=[\mathcal{A}_{0},...,\mathcal{A}_{i-1}]$,
$\mathcal{A}_{k}\subseteq \mathcal{V}(\mathcal{T}(z_{k})\backslash
z_{k})$ for $k\in \{0,...,i-1\}$ and any $\mathcal{A}_{k}$ is an
incomparable vertex subset.

\item $[\Phi]$ is a sequence of ordered sets of directed tree
homomorphisms - $[\Phi]=[\Phi_{0},...,\Phi_{i-1}]$ with
$\Phi_{k}=[\varphi_{x}]_{x\in \mathcal{A}_{k}}$, where
$\varphi_{x}\in \mathcal{H}om(\mathcal{T}(x), \mathcal{T}(g(x)))$.

\end{enumerate}

\end{theorem}

\begin{proof} Considering the action of $g$ on $f$-cycles
we get that the image of $Z$ is an $f$-cycle of length $l$, where
$l|i$. This cycle is bijectively determined by $g(z_{0})$. See
Theorem \ref{9}. Considering images of directed trees
consecutively increasing vertex heights we get $[\mathcal{A}]$.,
use Lemma \ref{4}. For any $k$ and $x\in \mathcal{A}_{k}$ the
subtree $\mathcal{T}(x)$ is independently homomorphically mapped
to $\mathcal{T}(g(x))$ and taking sequences of homomorphisms over
all $k$ and $x$ we get $[\Phi]$.
\end{proof}

\begin{remark} Thus to define an $f$-commuting function on an $f$-pseudocycle we need
\begin{enumerate}
\item to map the $f$-cycle to an $f$-cycle of appropriate length,

\item to define vertices whose inverse images with respect ro $f$
leave the directed cycle (vertices forming the sets
$\mathcal{A}_{k}$) by travelling backwards the edges of directed
trees,

\item to define graph homomorphisms for remaining subtrees.
\end{enumerate}

Note that an $f$-pseudocycle of cycle length $i$ can be mapped by
an $f$-commuting endofunction $g'$ to any $f$-pseudocycle of cycle
length $l$, $l|i$. To prove that the set of commuting functions is
nonempty we can take $g'$, which sends all vertices of positive
height to the $f$-cycle.
\end{remark}

\begin{remark} Given an $f$-commuting function $g$ we can
determine $g$-pseu\-do\-cyc\-les by first considering the
$g$-image of the set of $f$-cycles and then considering $g$-images
of the directed trees of $f$.

\end{remark}

\begin{example}\label{ex7}
Let $S=\{0,...,9\}$ and the $f$-graph be given in Figure 4 below
\newpage

$$
\xymatrix{
5\ar[d]&6\ar[d]&7\ar[dl]&&\\
4\ar[d]&1\ar[dr]&8\ar[d]&9\ar[dl]\\
0\ar[ur]&&2\ar[dl]& \\
&3\ar[ul]&&\\ }
$$
\begin{center}
\

Fig.4.  - the $f$-graph for Example \ref{ex7}
\end{center}

Let us find $|\mathcal{C}(f)|$. There is one $f$-cycle (0,1,2,3),
therefore $|\mathcal{C}(f)|=\sum\limits_{i=0}^{3}N_{0i}$, where
$N_{0i}$ is the number of $f$-commuting functions sending $0$ to
$i$. Furthermore, $N_{0i}=\prod\limits_{j=0}^{3}M_{ij}$, where
$M_{ij}$ is the number of possible ways to map the directed tree
$T_{j}$ (the full directed tree with root $j$) if $0$ is mapped to
$i$. Below we give the table for $M_{ij}$.

\begin{center}\begin{tabular}{ r|c|c|c|c| } \multicolumn{1}{r}{}{}
 &  \multicolumn{1}{c}{$j=0$}
 & \multicolumn{1}{c}{$j=1$}&\multicolumn{1}{c}{$j=2$}& \multicolumn{1}{c}{$j=3$}\\
\cline{2-5}
$i=0$ & $2$ & $3^2$ &$3^2$&$1$\\
\cline{2-5}
$i=1$ & $2$ & $3^2$ &$1^2$&$1$\\
\cline{2-5}
$i=2$ & $3$ & $1$ &$2^2$&$1$\\
\cline{2-5}
$i=3$ & $3$ & $2^2$ &$3^2$&$1$\\
\cline{2-5}
\end{tabular}

\

\

Fig.4.  - table of $M_{ij}$ values for Example \ref{ex7}
\end{center}

We have that $|\mathcal{C}(f)|=2\cdot 3^2\cdot 3^2+2\cdot
3^2+3\cdot 2^2+3\cdot 2^2\cdot 3^2=300$.
\end{example}

We now describe an enumerative combinatorics result - a formula
for counting graph homomorphisms between two pseudocycles.

\begin{theorem}\label{34}
Let $P=(T_{0},...,T_{m-1})$ be a pseudocycle of cycle length $m$
with the directed cycle $Z=(z_{0},...,z_{m-1})$. Let
$P'=(T'_{0},...,T'_{l-1})$ be a pseudocycle of cycle length $l$
and the directed cycle $(z'_{0},...,z'_{l-1})$, where $l|m$. Let a
graph homomorphism $g:P\rightarrow P'$ be defined by the triple
$[g(z_{0}), [\mathcal{A}], [\Phi]]$ as in Theorem \ref{11}.

Then

\begin{enumerate}

\item
$\mathcal{H}om(P,P')=\bigsqcup\limits_{k=0}^{l-1}\mathcal{H}om(z_{0},z'_{k})$,
where $\mathcal{H}om(z_{0},z'_{k})$ is the set of graph
homomorphisms $P\rightarrow P'$ sending $z_{0}$ to $z'_{k}$;

\item $Hom(z_{0},z'_{k})=\bigsqcup\limits_{[\mathcal{A}]\in
Inc(P)}\mathcal{H}om(z_{0},z'_{k},[\mathcal{A}])$, where
$\mathcal{H}om(z_{0},z'_{k},[\mathcal{A}])$ is the set of graph
homomorphisms $P\rightarrow P'$ sending $z_{0}$ to $z'_{k}$ with
the sequence of incomparable vertex subsets $[\mathcal{A}]$, the
disjoint union is taken over $Inc(P)$ - all possible choices of
sequences of incomparable subsets $[\mathcal{A}]$;

\item
$\mathcal{H}om(z_{0},z'_{k},[\mathcal{A}])=\bigotimes\limits_{i=0}^{m-1}\bigotimes\limits_{h\ge
1}\bigotimes\limits_{x\in
\mathcal{A}_{i}(h)}\mathcal{H}om(\mathcal{T}(x),\widetilde{T}'_{i+k-h+1(mod\
l)})$

\item
$|\mathcal{H}om(P,P')|=\sum\limits_{k=0}^{l-1}|\mathcal{H}om(z_{0},z'_{k})|=
\sum\limits_{k=0}^{l-1}\sum\limits_{[\mathcal{A}]\in
Inc(P)}|\mathcal{H}om(z_{0},z'_{k},[\mathcal{A}])|=$

$= \sum\limits_{k=0}^{l-1}\sum\limits_{[\mathcal{A}]\in Inc(P)}
\prod\limits_{i=0}^{m-1}\prod\limits_{h\ge 1}\prod\limits_{x\in
\mathcal{A}_{i}(h)}|\mathcal{H}om(\mathcal{T}(x),\widetilde{T}'_{i+k-h+1(mod\
l)})|$

\end{enumerate}
\end{theorem}

\begin{proof}
1. A homomorphism $P\rightarrow P'$ sends $z_{0}$ to $Z'$.

2. A homomorphism $P\rightarrow P'$ determines the sequence
$[\mathcal{A}]=[\mathcal{A}_{0},...,\mathcal{A}_{m-1}]$ uniquely.

3. A homomorphism $g:P\rightarrow P'$ is uniquely determined by
the sequence $[g|_{T_{0}},...,g|_{T_{m-1}}]$. Restrictions of
homomorphisms to directed trees of $P$ can be chosen
independently. Furthermore, for any $i$ this restriction
$g|_{T_{i}}$ is determined by
$\mathcal{A}_{i}=[\mathcal{A}_{i}(1),\mathcal{A}(2),...]$. For any
$x\in \mathcal{A}_{i}(h)$ the subtree $\mathcal{T}(x)$ is
independently homomorphically mapped to the root-truncated
directed tree $\widetilde{T'}_{i+k-h+1 (mod\ l)}$. The $-h+1$ term
in the index corresponds to travelling along the cycle backwards
$h-1$ steps.

4. The formula follows counting elements of $\mathcal{H}om(P,P')$
using the statement 3 of this theorem, applying the sum rule and
the product rule.
\end{proof}

\begin{remark} Summation variables in statement 4 of Theorem \ref{34}
can be swapped.

\end{remark}

\begin{example} \label{ex8}

Consider again the graph of Example \ref{ex7}. The vertex sets of
root-truncated full directed trees are:
$V_{0}=\mathcal{V}(\widetilde{T_{0}})=\{4,5\}$, $V_{1}=\{6,7\}$,
$V_{2}=\{8,9\}$, $V_{3}=\emptyset$. The incomparable vertex
subsets are: $Inc(T_{0})=2^{V_{0}}\backslash\ V_{0}$,
$Inc(T_{1})=2^{V_{1}}$, $Inc(T_{2})=2^{V_{2}}$,
$Inc(T_{3})=\emptyset$. $Inc(P)=Inc(T_{0})\times Inc(T_{1})\times
Inc(T_{2})\times Inc(T_{3})$.

In terms of Example \ref{ex7} and Theorem \ref{34} we have that
$$N_{0i}=\sum\limits_{\mathcal{[A]}\in
Inc(P)}|\mathcal{H}om(0,i,[\mathcal{A}])|.$$ Let us check using
Theorem \ref{34} that $N_{02}=12$ which coincides with the
computation in Example \ref{ex7}. Nonzero contributions can be
given by incomparable subsets $\emptyset,\{5\}$ of $Inc(T_{0})$,
subset $\emptyset$ of $Inc(T_{1})$, subsets
$\emptyset,\{8\},\{9\},\{8,9\}$ of $Inc(T_{2})$ and subset
$\emptyset$ of $Inc(T_{3})$. We have eight sequences of
incomparable vertex subsets: four sequences each contributing $1$
and four sequences each contributing $2$ (for example,
$[\{5\},\emptyset,\{8\}]$), thus $N_{02}=12$.

\end{example}

We finish this subsection with a few results in related extremal
combinatorics.

\paragraph{Minimal centralizer}\

Let $U_{m,t}$ be a weakly indecomposable directed graph isomorphic
to a functional graph having one directed cycle of length $m$ and
one directed tree which is a directed path of $t$ vertices outside
the cycle. See Fig.5 for an example.

\begin{center}$$
\xymatrix{
&&&\bullet\ar[r]&\bullet\ar[r]&\bullet\ar[dl]\\
\bullet\ar[r]&\bullet\ar[r]&\bullet\ar[r]&\bullet\ar[r]&\bullet\ar[ul]&\\
}
$$

Fig.5.  - the graph $U_{4,4}$
\end{center}

\begin{proposition} Let $S$ be a finite set, $|S|=n$. Then

\begin{enumerate}

\item $\min\limits_{f\in \mathcal{F}un(S)}|\mathcal{C}(f)|=n$;

\item $|\mathcal{C}(f)|=n$ iff $\Gamma(f)\simeq U_{m,n-m}$ with
$1\le m\le n$ or $\Gamma(f)\simeq Z_{1}\cup U_{m,n-1-m}$ with
$1\le m\le n-1$.

\end{enumerate}

\end{proposition}

\begin{proof} 1. First we prove that $|\mathcal{C}(f)|=n$, if $\Gamma(f)$ is weakly indecomposable.
If $\Gamma(f)\simeq Z_{n}$, then $\mathcal{C}(f)=n$. We show that
$\mathcal{C}(f)\ge n$, if $\Gamma(f)$ is weakly indecomposable. If
$\Gamma(f)$ has $t$ vertices outside the directed cycle, $t<n$,
then there are at least $(n-t)+n=n$ endomorphisms: 1) $n-t$
endomorphisms which rotate the cycle and map all tree vertices to
the cycle, 2) for each tree vertex $x$ of height $h$ there is at
least one endomorphism rotating the cycle $h-1$ steps forward with
$[\mathcal{A}]$ having one nonempty element $\{x\}$, e.g.
$[\mathcal{A}]=[...,\emptyset,\{x\},\emptyset,...]$. Therefore
$\mathcal{C}(f)=n$, if $\Gamma(f)$ is weakly indecomposable.

Let $\Gamma(f)$ have $k$ weakly connected components, $k\ge 0$,
having $n_{1},...,n_{k}$ vertices, $n_{i}>1$, and $n_{0}$ trivial
components (of vertex size $1$), $n_{0}\ge 0$. The set
$\Gamma(f)$-endomorphisms permuting the trivial components and
mapping independently each of the $k$ nontrivial components to
itself or to trivial components is a subset of
$\mathcal{E}nd(\Gamma(f))$. The $i$th component can be mapped to
itself in at most $n_{i}$ ways. We have that $|\mathcal{C}(f)|\ge
n_{0}!\cdot n_{1}...n_{k}+n_{0}k$. By arguments used in
Proposition \ref{41} it is shown that $|\mathcal{C}(f)|\ge
n_{0}+n_{1}+...+n_{k}=n$.

2. If $\Gamma(f)\simeq U_{m,n-m}$ with $1\le m\le n$ or
$\Gamma(f)\simeq Z_{1}\cup U_{m,n-1-m}$ with $1\le m\le n-1$, then
we directly check that $\mathcal{C}(f)=n$.

We have to prove the other implication.

It was just proved that $|\mathcal{C}(f)|\ge n$. If $\Gamma(f)$ is
weakly indecomposable with more than one directed tree or a
directed tree which is not a path, then by direct analysis it can
be shown that at least one new endomorphism can be constructed,
thus $|\mathcal{E}nd(\Gamma(f))|=|\mathcal{C}(f)|> n$.

Let $\Gamma(f)$ have $k$ weakly connected nontrivial components
$\Gamma_{1}$,...,$\Gamma_{k}$, $k\ge 0$, having $n_{1},...,n_{k}$
vertices, $n_{i}>1$, and $n_{0}$ trivial components (of vertex
size $1$), $n_{0}\ge 0$. Then
$|\mathcal{C}(f)|=|\mathcal{E}nd(\Gamma(f))|\ge
n_{0}!\prod\limits_{i=1}^{k}|\mathcal{E}nd(\Gamma_{i})|+n_{0}k$.
If for at least one $i$ $\Gamma_{i}\not\simeq U_{m,t}$, then
$|\mathcal{C}(f)|>n_{0}!\cdot n_{1}...n_{k}+n_{0}k\ge n$. Suppose
that $\Gamma_{i}\simeq U_{m_{i},t_{i}}$ for some $m_{i},t_{i}$,
for any $i$.

\subparagraph{Case $n_{0}=0$} As in Proposition \ref{41} it is
shown that $\mathcal{C}(f)>n_{1}+...+n_{k}$, if $k\ge 2$. If
$k=1$, then we must have $\Gamma(f)\simeq U_{m,t}$ for some $m,t$.

\subparagraph{Case $n_{0}=1$} Again
$\mathcal{C}(f)>n_{0}+n_{1}+...+n_{k}+k> n$, if $k\ge 2$. If
$k=1$, then we must have $\Gamma(f)\simeq Z_{1}\cup U_{m,t}$ for
some $m,t$.

\subparagraph{Case $n_{0}\ge 2$} In this case $|\mathcal{C}(f)|\ge
n_{0}!\cdot n_{1}...n_{k}+n_{0}k\ge n_{0}n_{1}...,n_{k}+n_{0}k$.
If $k\ge 1$, then $|\mathcal{C}(f)|\ge
n_{0}+n_{1}+...+n_{k}+n_{0}k>n$. If $k=0$ and $n_{0}\ge 3$, then
$|\mathcal{C}(f)|=n_{0}!>n_{0}=n$. The cases $n_{0}\in \{1,2\}$
were considered earlier. Thus in this case $|\mathcal{C}(f)|>n$
for all $f$.
\end{proof}

\paragraph{Maximal centralizer}\

Now we consider functions with maximal centralizers. Without extra
conditions on $f\in \mathcal{F}un(S)$ the problem of finding
$\max\limits_{f\in \mathcal{F}un(S)}|\mathcal{C}(f)|$ is trivial,
the identity function commutes with any endofunction. We impose a
condition on $\Gamma(f)$ - let it be weakly indecomposable.

Let $W_{m,t}$ be a weakly indecomposable directed graph isomorphic
to a functional graph having one directed cycle of length $m$ and
$t$ proper tree vertices having height $1$ and adjacent to one
cycle vertex. See Fig.6 for an example.

\begin{center}$$
\xymatrix{
\bullet\ar[dr]&\bullet\ar[d]&\bullet\ar[dl]&\\
\bullet\ar[r]&\bullet\ar[r]&\bullet\ar[r]&\bullet\ar[d] \\
\bullet\ar[ur]&\bullet\ar[u]&\bullet\ar[l]& \bullet\ar[l]\\
 }
$$

Fig.5.  - the graph $W_{6,5}$
\end{center}

\begin{proposition}\label{44} Let $S$ be a finite set, $|S|=n$. Let
$Z_{m}(S)$ be the set of $S$-endofunctions having one directed
cycle of length $m$, let $t=n-m$. Then

\begin{enumerate}

\item $\max\limits_{f\in Z_{m}(S)}|\mathcal{C}(f)|=m-1+(t+1)^t$;

\item $|\mathcal{C}(f)|=m-1+(t+1)^t$ iff $\Gamma(f)\simeq
W_{m,t}$.

\end{enumerate}
\end{proposition}

\begin{proof} 1. We show that for any $f\in Z_{m}(S)$ we have $|\mathcal{C}(f)|\le
|\mathcal{C}(w)|$, for some $w\in \mathcal{F}un(S)$ such that
$\Gamma(w)\simeq W_{m,t}$.

Let $S=S_{Z}\cup S_{T}$, where $S_{Z}=\{0,..,m-1\}$ is the only
$f$-cycle. Define $w\in \mathcal{F}un(S)$ such that $w(x)=f(x)$
for any $x\in S_{Z}$ and $w(y)=0$ for any $y\in S_{T}$. We see
that $\Gamma(w)\simeq W_{m,t}$.

We define a map $\delta:\mathcal{C}(f)\rightarrow \mathcal{C}(w)$
as follows. Let $g\in \mathcal{C}(f)$.

\subparagraph{Case $g(S_{T})\cap S_{T}=\emptyset$} Let
$(\delta(g))(x)=g(x)$ for any $x\in S_{Z}$ and
$(\delta(g))(y)\equiv g(0)-1(mod\ m)$ for any $y\in S_{T}$.
Informally, if $g$ sends all tree vertices to the cycle, then
$\delta(g)$ acts on the cycle as $g$ and acts on tree vertices in
such a way that $\delta(g)\in \mathcal{C}(f)$, all tree vartices
are mapped to the cycle.

\subparagraph{Case $g(S_{T})\cap S_{T}\neq\emptyset$} Let $(\delta
(g))(x)=x$ for any $x\in S_{Z}$ and $(\delta (g))(y)=g(y)$ for any
$y\in S_{T}$. Informally, if $g$ does not send all vertices to the
cycle then $\delta(g)$ fixes the cycle and maps such tree vertices
to tree vertices of height $1$ of $\Gamma(w)$. For any $x\in
S_{Z}$ we have $(\delta (g))(w(x))=w(x)=w(\delta(g))(x))$. For any
$y\in S_{T}$ we have $(\delta(g))(w(y))=(\delta(g))(0)=0$, on the
other hand $w((\delta(g))(y))=w(g(y))=0$. Thus $\delta(g)\in
\mathcal{C}(f)$.

The map $\delta$ is injective, because if $g(S_{T})\cap S_{T}\neq
\emptyset$, then the shift of the cycle and images of tree
vertices are uniquely determined. If $g(S_{T})\cap S_{T}=
\emptyset$, then the image of $\delta$ contains all shifts -
powers of $g$.

Injectivity of $\delta$ implies that $|\mathcal{C}(f)|\le
|\mathcal{C}(w)|$ and thus $\max\limits_{f\in
Z_{m}(S)}|\mathcal{C}(f)|=|\mathcal{C}(w)|$.

Let us compute $|\mathcal{C}(w)|$. There are $m-1$ nontrivial
(having nonidentity restriction on $S_{Z}$)
$\Gamma(w)$-endomorphisms $\widetilde{g}$ such that
$\widetilde{g}(S_{T})\cap S_{T}=\emptyset$. Now let us count the
number $N$ of $\Gamma(w)$-endomorphisms with identity restriction
on $S_{Z}$. Any such $S$-endofunction is determined by the pair
$[S',g']$, where $S'\subseteq S_{T}$ $g'\in
\mathcal{F}un(S_{T}\backslash S')$ ($S'$ are those vertices in
$S_{T}$, which are mapped to the cycle). By the sum and product
rule
$$N=\sum_{S'\subseteq S_{T}}|S_{T}|^{|S_{T}\backslash S'|}=\sum_{i=0}^{t}{t \choose
i}t^{i}=(t+1)^t.$$

Thus $|\mathcal{C}(w)|=m-1+N=m-1+(t+1)^t$.

2. We just determined that $\Gamma(f)\simeq W_{m,t}$ implies the
required size of $|\mathcal{C}(f)|$. We now show that if
$\Gamma(f)\not\simeq W_{m,t}$, then the map $\delta$ defined above
is not surjective.

If $\Gamma(f)$ has a directed tree with a vertex having height
bigger than $1$, then any permutation of its tree vertices that
increases height of a vertex cannot be obtained restricting a
$\Gamma(f)$-endomorphism.

Suppose all vertices of $\Gamma(f)$ have height at most $1$ and
there are at least two trees $T_{1}$ and $T_{2}$. If
$|\widetilde{T_{1}}|\ge 2$, then any $W_{m,p}$-endomorphism which
sends vertices of $T_{1}$ to two different trees can not be
obtained as an image of $\delta$. If any root-truncated tree has
one vertex and and there are at least three trees, then any
$W_{m,t}$-endomorphism which fixes one vertex and permutes a pair
of other vertices can not be obtained as an image of $\delta$. The
case then there are two vertices of height $1$ is proved by direct
computation. Thus if $\Gamma(f)\not\simeq W_{m,p}$, then $\delta$
is not surjective and in this case $|\mathcal{C}(f)|<m-1+(p+1)^p$.
\end{proof}

\begin{proposition}Let $S$ be a finite set, $|S|=n$. Let
$Z_{M}(S)$ be the set of $S$-endofunctions having cycles of
lengths belonging to the multiset $M=\{\{m_{1},...,m_{k}\}\}$,
$m_{i}\le m_{i+1}$ (one directed cycle of length $m_{i}$ for each
$i$). Let $t=n-\sum\limits_{i=1}^{k}m_{i}\ge 0$. Then

\begin{enumerate}

\item $\max\limits_{f\in
Z_{M}(S)}|\mathcal{C}(f)|=\Big(\prod\limits_{i=2}^{k}\sum\limits_{j:\
m_{j}|m_{i}}m_{j}\Big)\cdot \Big((t+1)^t-1+\sum\limits_{j:\
m_{j}=m_{1}}m_{j}\Big)$;

\item $|\mathcal{C}(f')|=\max\limits_{f\in
Z_{M}(S)}|\mathcal{C}(f)|$ iff $\Gamma(f')\simeq W_{m_{1},t}\cup
\Big( \bigcup\limits_{i=2}^{k}Z_{m_{i}}\Big)$.

\end{enumerate}

\end{proposition}

\begin{proof}

1., 2. If $\Gamma(f)\simeq W_{m_{1},t}\cup \Big(
\bigcup\limits_{i=2}^{k}Z_{m_{i}}\Big)$, then the formula
$$|\mathcal{C}(f)|=\Big(\prod\limits_{i=2}^{k}\sum\limits_{j:\
m_{j}|m_{i}}m_{j}\Big)\cdot \Big((t+1)^t-1+\sum\limits_{j:\
m_{j}=m_{1}}m_{j}\Big)$$ is checked by direct computation as
follows. Let $f_{0}$ be the function with the same multiset of
cycles $M$ and $t=0$, then using statement 3 of \ref{42} we get

\begin{equation}\label{43}
|\mathcal{C}(f_{0})|=\Big(\prod\limits_{i=1}^{k}\sum\limits_{j:\
m_{j}|m_{i}}m_{j}\Big).
\end{equation}
 By adding $t$ tree vertices of height $1$
to one vertex of a minimal length cycle one factor of \ref{43} for
the minimal $m_{i}$, say, $i=1$, changes to
$\Big((t+1)^t-1+\sum\limits_{j:\ m_{j}=m_{1}}m_{j}\Big)$, see
\ref{44}.

It is left to prove, that $|\mathcal{C}(f)|$ is maximal iff all
tree vertices of $V(f)$ are attached with height $1$ to one vertex
of a cycle of the minimal length. Any pseudoforest on the vertex
set $S$ having cycles of lengths in $M$ and $t$ tree vertices is
tranformed into $W_{m_{1},t}\cup \Big(
\bigcup\limits_{i=2}^{k}Z_{m_{i}}\Big)$ by a sequence of following
moves which increase the number of graph endomorphisms: 1)
transform each tree into a tree of type $W_{m_{i},t'}$ for some
$t'$, 2) all tree vertices are moved to one cycle of minimal
length, thus getting $W_{m_{1},t}$ and other cycles. All details
are not given.
\end{proof}

\begin{example} In Fig.6 we see the functional graph having cycles
of length $2,2,4$, three tree vertices and $1072=4\cdot 4\cdot
(4^3-1+4)$ commuting functions which is the maximal number for
this cycle length set.

\begin{center}$$
\xymatrix{
\bullet\ar[dr]&\bullet\ar[d]&\bullet\ar[dl]&&&\bullet\ar[r]&\bullet\ar[d]&\\
& \bullet\ar\ar@/_/[r] & \bullet\ar\ar@/_/[l]&\bullet\ar\ar@/_/[r] & \bullet\ar\ar@/_/[l]&\bullet\ar[u]&\bullet\ar[l]&\\
 }
$$

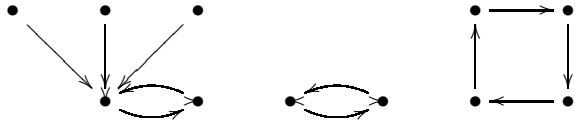
Fig.6.  - the graph $W_{2,3}\cup 2 Z_{3}\cup Z_{4}$
\end{center}

\end{example}

\subsection{Generalizations to functions on infinite sets}\label{20}

Results for finite sets can be transferred to the case of infinite
sets, where there are additional types of weakly connected
components of functional graphs.

\begin{lemma}\label{21} Let $S$ be a set, $f$ and $g$ - commuting $S$-endofunctions: $f g=g
f$. Let $L\le \Gamma(f)$ be an $f$-line, let $R\le \Gamma(f)$ be
an $f$-ray.

Then

\begin{enumerate}

\item $g(L)$ is an $f$-line, an $f$-cycle or an $f$-cycle with an
infinite directed tree - infinite directed path.

\item $g(R)$ is an $f$-ray, an $f$-cycle or an $f$-cycle with a
finite directed path.

\end{enumerate}

\end{lemma}

\begin{proof} Assume that $L=(V_{L},E_{L})$, where $V_{L}=\bigsqcup\limits_{a\in
\mathbb{Z}}x_{a}$, $E_{V}=\bigsqcup\limits_{a\in
\mathbb{Z}}[x_{a},x_{a+1}]$, $R=(V_{R},E_{R})$, where
$V_{R}=\bigsqcup\limits_{a\in \mathbb{N}}x_{a}$,
$E_{R}=\bigsqcup\limits_{a\in \mathbb{N}}[x_{a},x_{a+1}]$.

If $g$ is injective on $L$ (or $R$), then $g(L)$ (or $g(R)$) is an
$f$-line (or $f$-ray).

If $g$ is not injective on $L$ (or $R$), then there are two $L$
(or $R$) vertices $x_{m}$ and $x_{m+a}$, $a>0$, such that
$g(x_{m+a})=g(x_{m})$. It follows, that
$g(x_{m})=g(f^{a}(x_{m}))=f^{a}(g(x_{m}))$ and $g(x_{p})$ for all
$p\ge m$ belong to a finite $f$-cycle $Z$. If for all $x\in
\mathcal{V}(L)$ (or $x\in \mathcal{V}(R)$) we have that $g(x)\in
\mathcal{V}(Z)$, then $g(L)=Z$ (or $g(R)=Z$).

If there is $n<m$ such that $g(x_{n})$ does not belong to
$\mathcal{V}(Z)$, then $g(L)$ (or $g(R)$) is $Z$ with an infinite
(or finite) directed path having its root in $\mathcal{V}(Z)$.
\end{proof}

\begin{proposition} Let $S$ be a set, $f$ and $g$ - commuting $S$-endofunctions:
$f g=g f$. Let $X\le \Gamma(f)$ be an $f$-line or $f$-ray.

Then $g(X)$ contains a directed $f$-cycle iff $g$ is not injective
on $\mathcal{V}(X)$.

\end{proposition}

\begin{proof} If there is $k\in \mathbb{N}$ and $x,y\in S$ such that $f^{k}(x)=y$
and $g(x)=g(y)$, then $g(f^{k+a}(x))=f^{a}(gf^{k}(x))=f^{a}(g(y))$
for all $a\in \mathbb{N}$. It follows that the induced
$f$-subgraph having vertex set $\bigcup\limits_{k\ge 0}f^{k}(x)$
is a directed $f$-cycle.

If $g(X)$ contains an $f$-cycle $Z$ as a subgraph, then there is
$v\in \mathcal{V}(X)$ such that $f^{k}(g(v))=g(v)$ for some $k\in
\mathbb{N}$. It follows that $g(v)=g(f^{k}(v))$, hence $g$ is not
injective on $X$.
\end{proof}

\begin{lemma}\label{22} Let $S$ be a set, $f$ and $g$ - commuting $S$-endofunctions: $f g=g
f$. Let $P$ be an $f$-pseudoline (or $f$-pseudoray) with a
directed line (ray) $L$ (or $R$). $T$ is a directed tree of $P$
with the root $z\in \mathcal{V}(L)$ (or $z\in \mathcal{V}(R)$).

Then there is $\mathcal{A}\subseteq \mathcal{V}(T\backslash
\{z\})$ such that

\begin{enumerate}

\item[1)] there are no distinct elements $x,y\in \mathcal{A}$,
such that $x<y$ in the tree order;

\item[2)] $\phi(g(x))=1$ iff $x\in \mathcal{A}$;

\item[3)] for any $x\in \mathcal{A}$ if $x\le y$ then $g(x)$ and
$g(y)$ are in the same directed tree.
\end{enumerate}

\end{lemma}

\begin{proof} Similar to Lemma \ref{4}.
\end{proof}

\begin{theorem} Let $S$ be a set, $f$ and $g$ - commuting $S$-endofunctions:
$fg=gf$.

Then

\begin{enumerate}

\item the $g$-image of an $f$-pseudocycle of cycle length $m$ is
an $f$-pseudocycle of cycle length $l$, where $l|m$;

\item if $P$ is an $f$-pseudocycle with directed cycle
$Z=(z_{0},...,z_{m-1})$ and directed tree cycle
$T=(T_{0},...,T_{m-1})$, then the restriction
$g|_{\mathcal{V}(P)}$ is bijectively defined by the triple
$\tau=[g|_{Z}, [\mathcal{A}], [\Phi]]$, where

\begin{enumerate}



\item $[\mathcal{A}]=[\mathcal{A}_{0},...,\mathcal{A}_{m-1}]$,
where $\mathcal{A}_{i}\subseteq \mathcal{V}(T_{i}\backslash
z_{i})$ such that for any $a\in \mathcal{V}(P)$ we have that
$\phi(g(a))=1$ iff $a\in \mathcal{A}_{i}$ for some $i$ ($\phi$ is
meant with respect to $g(Z)$);

\item $[\Phi]=[\Phi_{0},...,\Phi_{m-1}]$, where
$\Phi_{i}=[\varphi_{i,x}]_{x\in \mathcal{A}_{i}}$, where

$\varphi_{i,x}\in \mathcal{H}om(\mathcal{T}(x),
\mathcal{T}(g(x)))$.

\end{enumerate}

\end{enumerate}
\end{theorem}

\begin{proof} Similar to Theorem \ref{11}.
\end{proof}

\begin{theorem} Let $S$ be a set, $f$ and $g$ - commuting $S$-endofunctions:
$fg=gf$. Let $P$ be an $f$-pseudoline (or $f$-pseudoray) with a
directed line $Z=(...,z_{0},z_{1},...)$ (or a directed ray
$Z=(z_{1},...)$), define $T_{i}=\mathcal{T}(z_{i})$.

Then

\begin{enumerate}

\item if $g|_{Z}$ is injective, then $g|_{P}$ is bijectively
defined by the sequence $$\tau=[g|_{Z}, [\mathcal{A}], [\Phi]],$$
where

\begin{enumerate}


\item $[\mathcal{A}]$ is the sequence $[\mathcal{A}_{i}]_{i\in
\mathbb{Z}}$ (or $[\mathcal{A}_{i}]_{i\in \mathbb{N}}$), where
$\mathcal{A}_{i}\subseteq \mathcal{V}(T_{i}\backslash z_{i})$ such
that for any $a\in \mathcal{V}(P)$ we have that $\phi(g(a))=1$ iff
$a\in \mathcal{A}_{i}$ ($h$ is meant with respect to $g(Z)$);

\item $[\Phi]$ is the sequence $[\Phi_{i}]_{i\in \mathbb{Z}}$ (or
$[\Phi_{i}]_{i\in \mathbb{N}}$), where
$\Phi_{i}=[\varphi_{i,x}]_{x\in \mathcal{A}_{i}}$ with
$\varphi_{i,x}\in \mathcal{H}om(\mathcal{T}(x),
\mathcal{T}(g(x)))$.

\end{enumerate}

\item If $g|_{Z}$ is not injective and $g(z_{p})=g(z_{q})$ with
$p<q$, then $g|_{P}$ is bijectively defined by the sequence
$\tau=[g|_{Z_{p,q}}, \mathcal{A}, [\Phi]]$, where

\begin{enumerate}

\item denote by $Z_{p,q}$ the induced subgraph of $Z$ with the
vertex set $\{z_{p},...,z_{q}\}$;

\item $[\mathcal{A}]=[\mathcal{A}_{i}]_{i\ge p}$, where
$\mathcal{A}_{i}\subseteq \mathcal{V}(\mathcal{T}(z_{i})\backslash
z_{i})$ such that $\phi(g(a))=1$ iff $a\in \mathcal{A}_{i}$
($\phi$ is meant with respect to $g(Z)$);

\item $[\Phi]=[\Phi_{i}]_{i\ge p}\Phi_{i}$ with
$\Phi_{i}=[\varphi_{i,x}]_{x\in \mathcal{A}_{i}}$, where
$\varphi_{i,x}\in \mathcal{H}om(\mathcal{T}(x),
\mathcal{T}(g(x)))$.

\end{enumerate}

\end{enumerate}
\end{theorem}

\begin{proof}1. If the injectivity condition holds for
$g(Z)$, then for each $z\in Z$ the restriction
$g|_{\mathcal{T}(z)}$ is determined by $\mathcal{A}\subseteq
\mathcal{V}(\mathcal{T}(z)\backslash z)$ containing vertices $a$
such that $\phi(g(a))=1$ and homomorphisms in
$\mathcal{H}om(\mathcal{T}(a),\mathcal{T}(g(a)))$ mapping the
remaining subtrees $\mathcal{T}(a)$ for each such $a$.

2. If the injectivity condition does not hold for $g(Z)$, then the
cyclic part of $g(Z)$ is determined by $g|_{Z_{p,q}}$, for each
$z_{n}\in Z$ with $n\ge p$ the restriction
$g|_{\mathcal{T}(z_{n})}$ is determined as in 1.

Use Lemma \ref{21} and Lemma \ref{22}.
\end{proof}

\subsection{Conclusion}

We have described endofunctions $g$ commuting with a given
endofunction $f$. Descriptions are given in terms of their
functional graphs, as homomorphisms of $f$-graphs, for $4$
subcases: 1) permutations commuting with a permutation, in this
case weakly connected components of $(f,g)$-graphs can be
interpreted as $g$-cycles, which permute $f$-cycles; 2)
permutations commuting with an arbitrary function, in this case
weakly connected components of $(f,g)$-graphs can be interpreted
as $g$-cycles, which permute $f$-pseudocycles sending directed
trees to isomorphic directed trees; 3) arbitrary functions
commuting with a permutation, in this case $(f,g)$-graphs can be
interpeted as $g$-pseudoforests with vertices being $f$-cycles; 4)
arbitrary functions commuting with an arbitrary function, this is
the most complex case: restrictions on $f$-cycles behave as in
case 3) and directed trees may be either mapped to cycles or leave
cycles and get mapped to directed trees. Results for finite sets
can be relatively straitforwardly generalized for arbitrary sets.
Future research may be stimulated by questions related to 1)
interpretation of graph-theoretic results in terms of functions,
matrices, operators etc., 2) enumerative and extremal
combinatorics, e.g. simplifications of the graph homomorphism
counting formula, Theorem \ref{34}, 3) graph structure of
functions satisfying other relations and 4) generalization of
these results to multivalued functions (mappings).

\end{document}